\documentclass{amsart}
\usepackage{amsmath}
\usepackage{amsthm}
\usepackage{amssymb}

\newtheorem{thm}{Theorem}
\newtheorem{lem}[thm]{Lemma}

\newtheorem{defi}[thm]{Definition}
\newtheorem{prop}[thm]{Proposition}
\newtheorem{rk}[thm]{Remark}

\newtheorem{alg}[thm]{Algorithm}

\newcommand{\rr}{{\mathbb{R}}}
\newcommand{\rd}{{\rr^3}}
\newcommand{\nn}{{\mathbb{N}}}
\newcommand{\sS}{{\mathbb{S}^2}}
\newcommand{\e}{\epsilon}
\newcommand{\vip}{\vskip.13cm}
\newcommand{\indiq}{\hbox{\rm 1}{\hskip -2.8 pt}\hbox{\rm I}}
\newcommand{\E}{\mathbb{E}}
\newcommand{\intot}{\int_0^t }
\newcommand{\intrd}{\int_{\rd}}

\newcommand{\intS}{\int_{\sS}}
\newcommand{\dd}{{\rm d}}

\newcommand{\cL}{{\mathcal{L}}}
\newcommand{\ctL}{\tilde{{\mathcal{L}}}}

\newcommand{\cA}{{\mathcal{A}}}
\newcommand{\cT}{{\mathcal{T}}}
\newcommand{\cB}{{\mathcal{B}}}

\newcommand{\cM}{{\mathcal{M}}}
\newcommand{\cP}{{\mathcal{P}}}

\newcommand{\vs}{v''}

\begin{document}

\title[Homogeneous Boltzmann equations]
{A recursive algorithm and a series expansion related to the homogeneous Boltzmann equation for hard potentials 
with angular cutoff}

\author{Nicolas Fournier}

\address{N. Fournier: Laboratoire de Probabilit\'es et Mod\`eles al\'eatoires, 
UMR 7599, UPMC, Case 188,
4 pl. Jussieu, F-75252 Paris Cedex 5, France.}

\email{nicolas.fournier@upmc.fr}



\begin{abstract}
We consider the spatially homogeneous Boltzmann equation for hard potentials with angular cutoff. 
This equation
has a unique conservative weak solution $(f_t)_{t\geq 0}$, once the initial condition $f_0$ with finite
mass and energy is fixed.
Taking advantage of the energy conservation, 
we propose a recursive algorithm that produces a $(0,\infty)\times\rr^3$ random variable
$(M_t,V_t)$ such that $\E[M_t\indiq_{\{V_t \in \cdot\}}]=f_t$. 
We also write down a series expansion of $f_t$. Although both the algorithm and the series expansion
might be theoretically interesting
in that they explicitly express $f_t$ in terms of $f_0$,
we believe that 
the algorithm is not very efficient in practice and that the series expansion is rather intractable.
This is a tedious extension to non-Maxwellian molecules of Wild's sum \cite{W} and of 
its interpretation by McKean \cite{MK,MK2}.
\end{abstract}

\maketitle

\section{Introduction}
\setcounter{equation}{0}

We consider a spatially homogeneous gas modeled by the Boltzmann equation:
the density $f_t(v)\geq 0$ of particles with velocity $v\in \rr^3$ at time $t\geq 0$ solves
\begin{align}\label{be}
\partial_t f_t(v) = \intrd \dd v^* \int_{\sS} \dd \sigma B(v-v^*,\sigma)
\big[f_t(v')f_t(v'^*) -f_t(v)f_t(v^*)\big],
\end{align}
where
\begin{align}\label{vprime}
v'=v'(v,v^*,\sigma)=\frac{v+v^*}{2} + \frac{|v-v^*|}{2}\sigma \quad \hbox{and}\quad
v'^*=v'^*(v,v^*,\sigma)=\frac{v+v^*}{2} -\frac{|v-v^*|}{2}\sigma.
\end{align}
The cross section $B$ is a nonnegative function given by physics. 
We refer to Cercignani \cite{c} and Villani \cite{v:h} for very complete books on the subject. 
We are concerned here with hard potentials with angular cutoff: the cross section satisfies
\begin{align}\label{ass}
\left\{\begin{array}{l}
B(v-v^*,\sigma) = |v-v^*|^\gamma b(\langle \frac{v-v^*}{|v-v^*|},\sigma\rangle) 
\quad  \hbox{for some $\gamma\in[0,1]$}\\ \\
\hbox{and some bounded measurable $b:[-1,1]\mapsto [0,\infty)$.}
\end{array}\right.
\end{align}
The important case where $\gamma=1$ and $b$ is constant corresponds to a gas of hard spheres.
If $\gamma=0$, the cross section is velocity independent and one talks about Maxwellian molecules
with cutoff.

\vip

We classically assume without loss of generality that the initial mass $\intrd f_0(v) \dd v=1$
and we denote by $e_0=\intrd |v|^2 f_0(v)\dd v>0$ the initial kinetic energy. 
It is then well-known, see Mischler-Wennberg \cite{MW}, that \eqref{be} has a unique weak solution
such that for all $t\geq 0$, $f_t$ is a probability density on $\rd$ with energy $\intrd |v|^2 f_t(v)\dd v=e_0$. 
Some precise statements are recalled in the next section.

\vip

In the whole paper, we denote, for $E$ a topological space, by $\cP(E)$ (resp. $\cM(E)$) 
the set of probability measures (resp. nonnegative measures)
on $E$ endowed with its Borel $\sigma$-field $\cB(E)$. For $v,v^*\in\rd$ and $\sigma \in \sS$, we
put
\begin{equation}\label{beta}
\beta_{v,v^*}(\sigma)=b(\langle {\textstyle{\frac{v-v^*}{|v-v^*|}}},\sigma\rangle) \quad \hbox{and we observe that} 
\quad \kappa= \intS \beta_{v,v^*}(\sigma)\dd\sigma
\end{equation}
does not depend on $v,v^*$ and is given by $\kappa=2\pi \int_0^\pi b(\theta)\sin\theta \dd\theta$. 
If $v=v^*$, then we have $v'=v'^*=v$, see \eqref{vprime},
so that the definition of $\beta_{v,v}(\sigma)$ is not important, we can e.g. set
$\beta_{v,v}(\sigma)=|\sS|^{-1}\kappa$.

\vip

In the rest of this introduction, we informally recall how \eqref{be} can be solved, 
in the case of Maxwellian molecules,
by using the Wild sum, we quickly explain its interpretation by McKean,
and we write down a closely related recursive simulation algorithm. We also recall that Wild's sum
can be used for theoretical and numerical analysis of Maxwellian molecules.
Then we briefly recall how the Wild sum and the algorithm can be easily extended to 
the case of any {\it bounded} cross section, by introducing fictitious jumps. Finally,
we quickly explain our strategy to deal with hard potentials with angular cutoff.

\subsection{Wild's sum}
Let us first mention that some introductions to Wild's sum and its probabilistic interpretation by McKean
can be found in the book of Villani \cite[Section 4.1]{v:h} and in Carlen-Carvalho-Gabetta \cite{ccg,ccg2}.
Wild \cite{W} noted that for Maxwellian molecules, i.e. when $\gamma=0$, so that the cross section
$B(v-v^*,\sigma)=\beta_{v,v^*}(\sigma)$ does not depend on the relative velocity,
\eqref{be} rewrites 
$$
\partial_t f_t(v)= \kappa Q(f_t,f_t)- \kappa f_t(v)
$$
where, for $f,g$ two probability densities on $\rd$,
$Q(f,g)(v)=\kappa^{-1}\int_{\sS} f(v')g(v'^*)\beta_{v,v^*}(\sigma)\dd \sigma$.

\vip

It holds that $Q(f,g)$ is also a probability density on $\rd$, that can be interpreted as
the law of $V'=(V+V^* + |V-V^*|\sigma)/2$, where $V$ and $V^*$ are two independent 
$\rd$-valued random variables with densities $f$ and $g$ and where $\sigma$ is, conditionally on
$(V,V^*)$, a $\kappa^{-1}\beta_{V,V^*}(\sigma)d\sigma$-distributed $\sS$-valued random variable.
Wild \cite{W} proved that given $f_0$, the solution $f_t$ to \eqref{be} is given by
\begin{equation}\label{w}
f_t=e^{-\kappa t} \sum_{n\geq 1} (1-e^{-\kappa t})^{n-1} Q_n(f_0),
\end{equation}
where $Q_n(f_0)$ is defined recursively by $Q_1(f_0)=f_0$ and, for $n\geq 1$, by
$$
Q_{n+1}(f_0)=\frac 1 n \sum_{k=1}^n Q(Q_k(f_0),Q_{n-k}(f_0)).
$$

McKean \cite{MK,MK2} provided an interpretation of the Wild sum in terms of binary trees, see
also Villani \cite{v:h} and Carlen-Carvalho-Gabetta \cite{ccg}. 
Let $\cT$ be the set of all discrete finite rooted ordered binary trees. 
By {\it ordered}, we mean that
each node of $\Upsilon \in \cT$ with two children has a {\it left} child and a {\it right} child. We denote by
$\ell(\Upsilon)$ the number of leaves of $\Upsilon \in \cT$. If $\circ$ is the trivial tree
(the one with only one node: the root), we set $Q_\circ(f_0)=f_0$. If now $\Upsilon \in \cT\setminus \{\circ\}$,
we put $Q_\Upsilon(f_0)=Q(Q_{\Upsilon_\ell}(f_0), Q_{\Upsilon_r}(f_0))$, where 
$\Upsilon_\ell$ (resp. $\Upsilon_r$) is the subtree
of $\Upsilon$ consisting of the left (resp. right) child of the root with its whole progeny.
Then \eqref{w} can be rewritten as
\begin{equation}\label{wwss}
f_t=e^{-\kappa t} \sum_{\Upsilon \in \cT} (1-e^{-\kappa t})^{|\Upsilon|-1} Q_\Upsilon(f_0).
\end{equation}
In words, \eqref{wwss} can be interpreted as follows. For each $\Upsilon \in \cT$, the term
$e^{-\kappa t} (1-e^{-\kappa t})^{|\Upsilon|-1}$ is the probability that 
a typical particle has $\Upsilon$ as (ordered) collision tree, while $Q_\Upsilon(f_0)$ is the density
of its velocity knowing that it has $\Upsilon$ as (ordered) collision tree.

\vip

Finally, let us mention a natural algorithmic interpretation of \eqref{be} closely related to \eqref{wwss}.
The {\it dynamical} probabilistic interpretation of Maxwellian molecules, initiated by Tanaka \cite{t},
can be roughly summarized as follows.
Consider a typical particle in the gas. Initially, its velocity $V_0$ is $f_0$-distributed. Then,
at {\it rate} $\kappa$, that is, after an Exp$(\kappa)$-distributed random time $\tau$, it collides with another
particle: its velocity $V_0$ is replaced by $(V_0+V^*_\tau + |V_0-V^*_\tau|\sigma)/2$, where $V^*_\tau$ 
is the velocity of an independent particle {\it undergoing the same process} (stopped at time $\tau$) and
$\sigma$ is a $\kappa^{-1}\beta_{V_0,V^*_\tau}(\sigma)d\sigma$-distributed $\sS$-valued random variable. 
Then, at {\it rate} $\kappa$, it collides again, etc. This produces a stochastic process $(V_t)_{t\geq 0}$ such
that for all $t\geq 0$, $V_t$ is $f_t$-distributed.

\vip

Consider now the following recursive algorithm.

\vip

{\tt Function velocity$(t)$:

.. Simulate a $f_0$-distributed random variable $v$, set $s=0$.

.. While $s<t$ do

.. .. simulate an exponential random variable $\tau$ with parameter $\kappa$,

.. .. set $s=s+\tau$,

.. .. if $s<t$, do

.. .. .. set $v^*=$velocity$(s)$,

.. .. .. simulate a $\kappa^{-1}\beta_{v,v^*}(\sigma)d\sigma$-distributed $\sS$-valued random variable $\sigma$,

.. .. .. set $v=(v+v^* +|v-v^*|\sigma)/2 $,

.. .. end if,

.. end while.

.. Return velocity$(t)=v$.
}

\vip

Of course, each new random variable is simulated independently of the previous ones. In particular,
line 7 of the algorithm, all the random variables required to produce {\tt $v^*=$velocity$(s)$}
are independent of all that has already been simulated.

\vip

Comparing the above paragraph and the algorithm, it appears clearly that {\tt velocity}$(t)$
produces a $f_t$-distributed random variable. We have never seen this fact written precisely as it is here,
but it is more or less well-known folklore. In the present paper, we will prove such a fact, in a slightly
more complicated situation.

\vip

In spirit, the algorithm produces a binary ordered tree: each time the recursive function calls itself, 
we add a branch (on the right). So it is closely related to \eqref{wwss} and, actually, one can get
convinced that {\tt velocity$(t)$} is precisely an algorithmic interpretation of \eqref{wwss}.
But entering into the details would lead us to tedious and technical explanations.

\subsection{Utility of Wild's sum}

The Wild sum has often been used for numerical computations: one simply cutoffs \eqref{w} at some
well-chosen level and, possibly, adds a Gaussian distribution with adequate mean and covariance matrix
to make it have the desired mass and energy.
See Carlen-Salvarini \cite{cs} for a very precise study in this direction.
And actually, Pareschi-Russo \cite{pr} also managed to use the Wild sum, among many other things,
to solve numerically the {\it inhomogeneous} Boltzmann equation for non Maxwellian molecules.

\vip

A completely different approach is to use a large number $N$ of times 
the perfect simulation algorithm previously described
to produce some i.i.d. $f_t$-distributed random variables $V^1_t,\dots,V^N_t$,
and to approximate $f_t$ by $N^{-1}\sum_1^N \delta_{V^i_t}$. We believe that this is not very efficient in practice,
especially when compared to the use of a classical interacting particle system in the spirit of Kac 
\cite{k}, see e.g. \cite{fsm}. The main reason is that the computational cost of the above {\it perfect} 
simulation algorithm
increases exponentially with time, while the one of Kac's particle system increases linearly. 
So the cost to remove the bias is disproportionate.
See \cite{fg} for such a study concerning the Smoluchowski equation, which has the same structure
(at the rough level) as the Boltzmann equation.

\vip

The Wild sum has also been intensively used to study the rate of approach to equilibrium of Maxwellian molecules.
This was initiated by McKean \cite{MK}, with more recent studies by 
Carlen-Carvalho-Gabetta \cite{ccg,ccg2}, themselves followed by  Dolera-Gabetta-Regazzini
\cite{dgr} and many other authors.

\subsection{Bounded cross sections}
If $B(v-v^*,\sigma)=\Phi(|v-v^*|)b(\langle \frac{v-v^*}{|v-v^*|},\sigma\rangle)$ with $\Phi$ bounded, e.g. by $1$,
we can introduce fictitious jumps
to write \eqref{be} as $\partial_t f_t = \kappa Q(f_t,f_t) - \kappa f_t$, with 
$Q(f,g)= \kappa^{-1} \int_\sS\int_0^1 f(v'')g(v''^*) \dd a \beta_{v,v^*}(\sigma)\dd \sigma$, where
$v''= v + [v'-v]\indiq_{\{a \leq \Phi(|v-v^*|)\}}$ and something similar for $v''^*$.
Hence all the previous study directly applies, but the resulting Wild sum does not seem to
allow for a precise study the large time behavior of $f_t$, because it leads to intractable
computations.

\subsection{Hard potentials with angular cutoff}

Of course, the angular cutoff (that is, we assume that $\kappa < \infty$)
is crucial to hope for a perfect simulation algorithm
and for a series expansion in the spirit of Wild's sum. Indeed, $\kappa=\infty$
implies that a particle is subjected to infinitely many collisions on each finite time interval.
So our goal is to extend, at the price of many complications, the algorithm and series expansion
to hard potentials with cutoff. Since the cross section  
is unbounded in the relative velocity variable,
some work is needed.

\vip

We work with weak forms of PDEs for simplicity. First, it is classical, see e.g. \cite[Section 2.3]{v:h} 
that a family 
$(f_t)_{t\geq 0}\subset \cP(\rd)$ is a weak solution to \eqref{be} if it 
satisfies, for all reasonable $\phi\in C_b(\rd)$,
$$
\frac{d}{dt} \intrd \phi(v)f_t(\dd v)=\intrd\intrd \int_{\sS} |v-v^*|^\gamma [\phi(v')-\phi(v)] \beta_{v,v^*}(\sigma) 
\dd \sigma f_t(\dd v^*) f_t(\dd v).
$$
As already mentioned, one also has $\intrd |v|^2 f_t(\dd v) = e_0$, so that
$g_t(\dd v)= (1+e_0)^{-1}(1+|v|^2)f_t(\dd v)$ belongs to $\cP(\rd)$ for all $t\geq 0$. A simple computation
shows that,  for all reasonable $\phi\in C_b(\rd)$,
$$
\frac{d}{dt} \intrd \phi(v)g_t(\dd v)=\intrd\intrd \int_{\sS} \frac{(1+e_0)|v-v^*|^\gamma} {1+|v^*|^2}
\Big[\frac{1+|v'|^2}{1+|v|^2}\phi(v')-\phi(v)\Big] \beta_{v,v^*}(\sigma) \dd 
\sigma g_t(\dd v^*) g_t(\dd v).
$$
This equation enjoys the pleasant property that the {\it maximum rate of collision}, given by 
$\Lambda_0(v)=\kappa \sup_{v^*\in\rd}
(1+|v^*|^2)^{-1}(1+e_0)|v-v^*|^\gamma$, is finite. Hence, up to some fictitious jumps,
one is able to predict when a particle will collide from its sole velocity, knowing nothing
of the {\it environment} represented by $g_t$. The function
$\Lambda_0$ is not bounded as a function of $v$, since it resembles $1+|v|^\gamma$, 
but, as we will see, this it does not matter too much.
On the contrary, the presence of $(1+|v|^2)^{-1}(1+|v'|^2)$ in front of the gain term is problematic. It means
that, in some sense, particles are not all taken into account equally. To overcome this problem, 
we consider an equation with an additional {\it weight} variable $m\in (0,\infty)$. So we search for an equation,
resembling more the Kolmogorov forward equation of a nonlinear Markov process,
of which the solution $(G_t)_{t\geq 0} \subset \cP((0,\infty)\times\rd)$ would be such that
$\int_0^\infty m G_t(\dd m, \dd v)=g_t(\dd v)$ for all times. Then, one would recover the solution to \eqref{be}
as $f_t(\dd v) = (1+e_0)(1+|v|^2)^{-1}\int_0^\infty m G_t(\dd m, \dd v)$. All this is doable and was our initial 
strategy. However, we then found a more direct way to proceed: taking advantage of the energy conservation,
it is possible to build an equation of which the
solution $(F_t)_{t\geq 0} \subset \cP((0,\infty)\times\rd)$ is such that 
$\int_0^\infty m F_t(\dd m, \dd v)=f_t(\dd v)$ for all $t\geq 0$. And this equation is of the form
\begin{align*}
&\frac{d}{dt} \int_{(0,\infty)\times\rd}\Phi(m,v)F_t(\dd m,\dd v)=\int_{(0,\infty)\times\rd}\int_{(0,\infty)\times\rd}
\int_{\sS\times[0,1]} \Lambda(v)
\Big[\Phi(m'',\vs )  -\Phi(m,v)\Big] \\
&\hskip9cm \beta_{v,v^*}(\sigma) \dd a \dd \sigma F_t(\dd m^*,\dd v^*)F_t(\dd m,\dd v).
\end{align*}
Here $\Lambda$ is any (explicit) function dominating $\Lambda_0$,
the additional variable $a$ is here to allow for fictitious jumps, and the post-collisional characteristics
$(m'',v'')$, depending on $m,v,m^*,v^*,\sigma,a$ are precisely defined in the next section. 
The {\it perfect} simulation algorithm 
for such an equation is almost as simple as the one previously described, except that the rate of collision
$\Lambda(v)$ now depends on the state of the particle. On the contrary, this state-dependent rate 
complicates subsequently the series expansion because the time and phase variables do not separate anymore.

\subsection{Plan of the paper}

In the next section, we expose our main results: we 
introduce an equation with an additional variable $m$, state that this equation has a unique
solution $F_t$ that, once integrated in $m$, produces the solution $f_t$ to \eqref{be}.
We then propose an algorithm that perfectly simulates an $F_t$-distributed random variable,
we write down a series expansion for $F_t$ in the spirit of \eqref{wwss} and discuss briefly
the relevance of our results. The proofs are then handled: the algorithm is studied in Section \ref{sa},
the series expansion established in Section \ref{sps},
well-posedness of the equation solved by $F_t$ is checked in Section
\ref{swp}, and the link between $F_t$ and $f_t$ shown in Section \ref{sr}.

\section{Main results}

\subsection{Weak solutions}
We use a classical definition of weak solutions, see e.g. \cite[Section 2.3]{v:h}.

\begin{defi}\label{dfws}
Assume \eqref{ass} and recall \eqref{beta}. 
A measurable family $f=(f_t)_{t\geq 0} \subset \cP(\rr^3)$
is a weak solution to (\ref{be}) if for all $t\geq 0$, 
$\intrd |v|^2 f_t(\dd v)= \intrd |v|^2 f_0(\dd v)<\infty$ and for all $\phi \in C_b(\rr^3)$,
\begin{align}\label{ws}
\intrd \phi(v)f_t(\dd v) = \intrd \phi(v)f_0(\dd v) + \intot
\intrd \intrd  \cA \phi(v,v^*) f_s(\dd v^*)f_s(\dd v)\dd s,
\end{align}
where $\cA\phi(v,v^*)= |v-v_*|^\gamma \intS[\phi(v'(v,v^*,\sigma))-\phi(v)]\beta_{v,v^*}(\sigma) 
\dd \sigma$.
\end{defi}

Everything is well-defined in \eqref{ws} by boundedness of mass and energy and since 
$|\cA\phi(v,v^*)|\leq 2 \kappa ||\phi||_\infty |v-v^*|^\gamma
\leq 2 \kappa ||\phi||_\infty (1+|v|^2)(1+|v^*|^2)$.

\vip

For any given $f_0 \in \cP(\rr^3)$ such that $\intrd |v|^2 f_0(\dd v)<\infty$, 
the existence of unique weak solution starting from $f_0$
is now well-known. See Mischler-Wennberg \cite{MW} when $f_0$ has a density
and Lu-Mouhot \cite{LM} for the general case. 
Let us also mention that the conservation assumption
is important in Definition \ref{dfws}, 
since Wennberg \cite{Wenn} proved that there also solutions with increasing energy.

\subsection{An equation with an additional variable}\label{tlf}
We fix $e_0>0$ and define, for $v\in \rd$,
$$
\Lambda (v) = (1+e_0)(1+|v|^\gamma)\in [1,\infty) \quad \hbox{which satisfies}\quad
\Lambda(v) \geq \sup_{v^* \in \rd} \frac{(1+e_0)|v-v^*|^\gamma}{1+|v^*|^2}.
$$
For $v,v^*\in\rd$ and $z=(\sigma,a) \in H=\sS \times [0,1]$, we put
\begin{gather*}
\vs(v,v^*,z) =v+ \Big[v'(v,v^*,\sigma)-v\Big]\indiq_{\{a\leq q(v,v^*)\}} \in \rd,\\
\hbox{where}\quad 
q(v,v^*)=\frac{(1+e_0)|v-v^*|^\gamma}{(1+|v^*|^2)\Lambda(v)}\in [0,1].
\end{gather*}
We also introduce $E=(0,\infty)\times \rr^3$ and, for $y=(m,v)$ and $y^* =(m^*,v^*)$ in $E$ and $z\in H$,
$$
h(y,y^*,z)=\Big(\frac{mm^*(1+|v^*|^2)}{1+e_0},\vs(v,v^*,z)\Big)\in E \quad\hbox{and}\quad \Lambda(y)=\Lambda(v)
\in [1,\infty)
$$
with a small abuse of notation. We finally consider, for $y=(m,v)$ and $y^* =(m^*,v^*)$ in $E$,
$$
\nu_{y,y^*}(\dd z)=\beta_{v,v^*}(\sigma)\dd\sigma \dd a
$$ 
which is a measure on $H$ with total mass $\nu_{y,y^*}(H)=\kappa$, see \eqref{ass}.

\begin{defi}\label{GE}
Assume \eqref{ass}. A measurable family $F=(F_t)_{t\geq 0} \subset \cP(E)$ is said to solve (A) if
for all $T>0$, $\sup_{[0,T]} \int_E |v|^\gamma F_t(\dd m,\dd v)<\infty$,
and for all $\Phi \in C_b(E)$, all $t\geq 0$, 
\begin{align}\label{gew}
\int_E \Phi(y)F_t(\dd y)=\int_E \Phi(y)F_0(\dd y)+\intot \int_E\int_E \cB\Phi(y,y^*) F_s(\dd y^*)
F_s(\dd y)\dd s,
\end{align}
where $\cB\Phi(y,y^*)=\Lambda(y)\int_H[\Phi(h(y,y^*,z))-\Phi(y)]\nu_{y,y^*}(\dd z)$.
\end{defi}

All is well-defined in \eqref{gew}
thanks to the conditions on $F$ and since 
$|\cB\Phi(y,y^*)|\leq 2\kappa ||\Phi||_\infty  \Lambda(y)\leq C_\Phi(1+|v|^\gamma)$ (with the notation
$y=(m,v)$).
As already mentioned in the introduction, 
the important point is that the function $\Lambda$ does not depend on $y^*$.
Hence a particle, when in the state $y$, jumps at rate $\kappa \Lambda(y)$, independently of everything else.

\begin{prop}\label{beeu}
Assume \eqref{ass}. For any $F_0\in \cP(E)$ such that $\int_E |v|^\gamma F_0(\dd m,\dd v)<\infty$, 
(A) has exactly one solution $F$
starting from $F_0$. 
\end{prop}

We will also verify the following estimate.

\begin{rk}\label{dec}
Assume \eqref{ass}. A solution to (A) satisfies
$\sup_{t\geq 0}\int_E |v|^2 F_t(\dd m,\dd v)=\int_E |v|^2 F_0(\dd m,\dd v)$.
\end{rk}

Finally, the link with the Boltzmann equation is as follows.

\begin{prop}\label{relat}
Assume \eqref{ass}. Let $F_0\in \cP(E)$ such that 
$\int_E [|v|^\gamma+m(1+|v|^{2+2\gamma})] F_0(\dd m,\dd v)<\infty$ and let 
$F$ be the solution to (A). Introduce, for each $t\geq 0$, the nonnegative measure $f_t$ on $\rr^3$
defined by $f_t(A)=\int_E m \indiq_{\{v \in A\}}F_t(\dd m, \dd v)$ for all $A\in \cB(\rr^3)$.
If $f_0 \in \cP(\rr^3)$ and if the quantity $e_0$ used to define the coefficients of (A) is
precisely $e_0=\int_{\rr^3} |v|^2 f_0(\dd v)$, then
$(f_t)_{t\geq 0}$ is the unique weak solution to \eqref{be} starting from $f_0$.
\end{prop}

\subsection{A perfect simulation algorithm}

We consider the following procedure.

\begin{alg}\label{algo}
Fix $e_0>0$ and $F_0 \in \cP(E)$. For any $t\geq 0$ we define the following recursive function,
of which the result is some $E\times \nn$-valued random variable.

\vip

{\tt function (value$(t)$,counter$(t)$):

.. Simulate a $F_0$-distributed random variable $y$, set $s=0$ and $n=0$.

.. While $s<t$ do

.. .. simulate an exponential random variable $\tau$ with parameter $\kappa \Lambda(y)$,

.. .. set $s=s+\tau$,

.. .. if $s<t$, do

.. .. .. set $(y^*,n^*)=$(value$(s)$,counter$(s)$),

.. .. .. simulate $z\in H$ with law $\kappa^{-1}\nu_{y,y^*}$, 

.. .. .. set $y=h(y,y^*,z)$,

.. .. .. set $n=n+n^*+1$,

.. .. end if,

.. end while.

.. Return value$(t)=y$ and counter$(t)=n$.
}
\end{alg}

Of course, each time a new random variable is simulated, we implicitly assume that it is
independent of everything that has already been simulated. In particular,
line 7 of the procedure, all the random variables used to produce
{\tt $(y^*,n^*)=$(value$(s)$,counter$(s)$)} are independent of all the random variables already simulated.
By construction, {\tt counter$(t)$} is precisely the number of times the recursive function
calls itself.

\begin{prop}\label{mr}
Assume \eqref{ass}. Fix $F_0\in \cP(E)$ such that $\int_E |v|^\gamma F_0(\dd m,\dd v)<\infty$ and fix $t\geq 0$. 
Algorithm \ref{algo} a.s. stops and thus produces a couple $(Y_t,N_t)$ of random variables. 
The $E$-valued random variable
$Y_t$ is $F_t$-distributed, where $F$ is the solution to (A) starting from $F_0$.
The $\nn$-valued random variable $N_t$ satisfies $\E[N_t]\leq 
\exp(\kappa \int_0^t \int_E \Lambda(y) F_s(\dd y) ds)-1$.
\end{prop}

\subsection{A series expansion}

We next write down a series expansion of $F_t$, the solution to (A), in the spirit of Wild's sum \eqref{wwss}.
Unfortunately, the expressions are more complicated,
because the time ($t\geq 0$) and phase ($y\in E$) variables do not separate. This is due to the
fact that the jump rate $\Lambda$ depends on the state of the particle. 

\vip

For $F,G$ in $\cM(E)$, we define $Q(F,G) \in \cM(E)$ by, for all Borel subset $A\subset E$,
$$
Q(F,G)(A)= \int_E\int_E\int_H \Lambda(y) \indiq_{\{h(y,y^*,z)\in A\}} \nu_{y,y^*}(\dd z) G(\dd y^*)F(\dd y).
$$
Observe that \eqref{gew} may be written, at least formally,
$\partial_t F_t=-\kappa \Lambda F_t + Q(F_t,F_t)$, provided $F_t$ is a probability measure on $E$ 
for all $t\geq 0$. Also, note that $Q(F,G)\neq Q(G,F)$ in general.

\vip

For $J\in \cM(\rr_+\times E)$, consider the measurable family $(\Gamma_t(J))_{t\geq 0}\subset \cM(E)$ defined by
$$
\Gamma_t(J)(A)=\intot \int_E \indiq_{\{y \in A\}} e^{-\kappa \Lambda(y) (t-s)}J(\dd s,\dd y)
$$
for all Borel subset $A\subset E$.

\vip

We finally consider the set $\cT$ of all finite binary (discrete) ordered trees: such a tree is constituted of
a finite number of nodes, including the root, each of these nodes having either $0$ or two children
(ordered, in the sense that a node having two children has a {\it left} child and a {\it right} child).
We denote by $\circ \in \cT$ the trivial tree, composed of the root as only node.

\begin{prop}\label{psum}
Assume \eqref{ass}. Let $F_0 \in \cP(E)$ such that $\int_E |v|^\gamma F_0(\dd m,\dd v)<\infty$. 
The unique solution
$(F_t)_{t\geq 0}$ to (A) starting from $F_0$ is given by
$$
F_t = \sum_{\Upsilon \in \cT} \Gamma_t(J_\Upsilon(F_0)),
$$
with $J_\Upsilon(F_0) \in \cM(\rr_+\times E)$ defined by induction:
$J_\circ(F_0)(\dd t,\dd x)=\delta_0(\dd t) F_0(\dd x)$ and, if $\Upsilon \in \cT\setminus\{\circ\}$,
$$
J_\Upsilon(F_0)(\dd t,\dd x)=\dd t Q(\Gamma_t(J_{\Upsilon_\ell}(F_0)),\Gamma_t(J_{\Upsilon_r}(F_0)))(\dd x),
$$ 
where $\Upsilon_\ell$ (resp. $\Upsilon_r$) is the subtree
of $\Upsilon$ consisting of the left (resp. right) child of the root with its whole progeny.
\end{prop}

We will prove this formula by a purely analytic method.
We do not want to discuss precisely its connection with Algorithm \ref{algo}, but
let us mention that in spirit, the algorithm produces a (random) ordered tree $\Upsilon_t$ of interactions 
together with the
value of $Y_t$, and that $\Gamma_t(J_\Upsilon(F_0))$ can be interpreted as the probability distribution of 
$Y_t$ restricted to the event that $\Upsilon_t=\Upsilon$.

\subsection{Conclusion}

Fix $f_0\in \cP(\rr^3)$ such that $\int_{\rr^3} |v|^{2+2\gamma}f_0(\dd v)<\infty$ and set 
$F_0= \delta_1 \otimes f_0 \in \cP(E)$, which satisfies
$\int_E [|v|^\gamma+m(1+|v|^{2+2\gamma})] F_0(\dd m,\dd v)
=\intrd (1+|v|^\gamma+|v|^{2+2\gamma}) f_0(\dd v)<\infty$.

\vip

(a) Gathering Propositions \ref{mr} and \ref{relat}, we find that
Algorithm \ref{algo} used with $e_0=\int_{\rr^3} |v|^{2}f_0(\dd v)$ and with 
$F_0$ produces a random variable $(Y_t,N_t)$, with $Y_t=(M_t,V_t)$ such that
$\E[M_t \indiq_{\{V_t \in A\}}]=f_t(A)$ for all $A \in \cB(\rr^3)$,
where $f$ is the unique weak solution to \eqref{be} starting from $f_0$.
Also, the mean number of iterations $\E[N_t]$ is bounded by $\exp[\kappa(1+e_0)(1+e_0^{\gamma/2}) t]-1$.

\vip

Indeed, we know from Proposition \ref{mr} that $\E[N_t]\leq 
\exp(\kappa \int_0^t \int_E \Lambda(y) F_s(\dd y) ds)-1$. But we have
$\int_E \Lambda(y)F_t(\dd y)=(1+e_0) \int_E (1+|v|^\gamma)F_t(\dd m,\dd v)
\leq (1+e_0)(1+ (\int_E |v|^2 F_t(\dd m,\dd v))^{\gamma/2})$, which is smaller than
$(1+e_0)(1+ (\int_E |v|^2 F_0(\dd m,\dd v))^{\gamma/2})=(1+e_0)(1+e_0^{\gamma/2})$ by Remark \ref{dec}.

\vip

(b) Gathering Propositions \ref{psum} and \ref{relat}, we conclude that
for all $t\geq 0$, all Borel subset $A\subset \rd$, we have 
$f_t(A)=\sum_{\Upsilon \in \cT} \int_E m \indiq_{\{v \in A\}}\Gamma_t(J_\Upsilon(F_0))(\dd m,\dd v)$.

\subsection{Discussion}
It might be possible to prove Proposition \ref{relat} assuming only that $F_0 \in \cP(E)$
satisfies $\int_E m(1+|v|^2) F_0(\dd m,\dd v)<\infty$ instead of
$\int_E m(1+|v|^{2+2\gamma}) F_0(\dd m,\dd v)<\infty$, since the Boltzmann equation \eqref{be}
is known to be well-posed as soon as the initial energy is finite, see \cite{MW,LM}.
However, it would clearly be more difficult and our condition is rather harmless.

\vip

Observe that (A) is well-posed under the condition that $F_0 \in\cP(E)$
satisfies $\int_E |v|^\gamma F_0(\dd m,\dd v)$, which does not at all imply that 
$e_0=\int_E m|v|^2 F_0(\dd m,\dd v)<\infty$. But, recalling that the $e_0$ has to be finite
for the coefficients of (A) to be well-defined, this is not very interesting.

\vip

The series expansion of Proposition \ref{psum} is of course much more complicated than the original
Wild sum, since (a) 
we had to add the variable $m$, (b) we had to introduce fictitious jumps, (c) time and space do not separate.
So it is not clear whether the formula can be used theoretically or numerically. However, it provides an
explicit formula expressing $f_t$ as a (tedious) function of $f_0$.

\vip

Algorithm \ref{algo} is extremely simple.
Using it a large number of times, which produces some i.i.d. sample
$(M^i_t,V^i_t)_{i=1,\dots,N}$, we may approximate $f_t$ by $N^{-1}\sum_1^N M^i_t \delta_{V^i_t}$. 
For a central limit
theorem to hold true, one needs $\E[M_t^2]=\int_E m^2F_t(\dd m,\dd v)$ to be finite. 
We do not know if this holds true, although
we have some serious doubts. Hence the convergence of this Monte-Carlo approximation may be much
slower that $N^{-1/2}$. The main interest of Algorithm \ref{algo} is thus theoretical.

\section{The algorithm}\label{sa}

Here we prove Proposition \ref{mr}. We fix $F_0 \in \cP(E)$ such that
$\int_E |v|^\gamma F_0(\dd m,\dd v)<\infty$,
which implies that $\int_E \Lambda(y) F_0(\dd y)<\infty$.
When Algorithm \ref{algo} never stops, we take the convention that it returns 
{\tt (value$(t)$,counter$(t)$)}=$(\triangle,\infty)$, where $\triangle$ is 
a cemetery point. For each $t\geq 0$, we denote by $G_t \in \cP((E\times \nn)\cup\{(\triangle,\infty)\})$
the law of the random variable produced by Algorithm \ref{algo}. 
Also, for $y \in E$, $n\in \nn$ and $z\in H$, we take the conventions that $h(y,\triangle,z)=\triangle$
and $n+\infty+1=\infty$. We arbitrarily define, for $y\in E$, $\nu_{y,\Delta}(\dd z)=|\sS|^{-1}\kappa \dd \sigma
\dd a$.

\vip

{\it Step 1.} We now consider the following procedure. It is an {\it abstract} procedure, because
it assumes that for each $t\geq 0$, one can simulate a random variable with law $G_t$
and because the instructions are repeated {\it ad infinitum} if the cemetery point is
not attained.

\vip

{\tt Simulate a $F_0$-distributed random variable $y$, set $s=0$ and $n=0$.

While $y \neq \triangle$ do ad infinitum

.. simulate an exponential random variable $\tau$ with parameter $\kappa \Lambda(y)$,

.. set $Y_t=y$ and $N_t=n$ for all $t \in [s,s+\tau)$,

.. set $s=s+\tau$,

.. set $(y^*,n^*)=$(value$(s)$,counter$(s)$), with $(y^*,n^*)=(\triangle,\infty)$ if it never stops,

.. simulate $z\in H$ with law $\kappa^{-1}\nu_{y,y^*}$,

.. set $y=h(y,y^*,z)$,

.. set $n=n+n^*+1$,

end while.

If $s<\infty$, set $Y_{t}=\Delta$ and $N_t=\infty$ for all $t\geq s$.
}

\vip

Observe that in the last line, we may have $s<\infty$ either because after a finite
number of steps, the simulation of $(y^*,n^*)$ with law $G_s$ has produced $(\triangle,\infty)$,
or because we did repeat the loop 
{\it ad infinitum}, but the increasing process $N$ became infinite in finite time.

\vip

At the end, this produces a process $(Y_t,N_t)_{t\geq 0}$ and one easily gets convinced that
for each $t\geq 0$, $(Y_t,N_t)$ is $G_t$-distributed. Indeed, if one extracts from the above
procedure only what is required to produce $(Y_t,N_t)$ (for some fixed $t$), one precisely re-obtains 
Algorithm \ref{algo} if $(Y_t,N_t)\neq (\triangle,\infty)$ 
(and in this case Algorithm \ref{algo} stops),
while $(Y_t,N_t)= (\triangle,\infty)$ implies that Algorithm \ref{algo} never stops.

\vip

By construction, the process $(Y_t,N_t)_{t\geq 0}$ 
is a time-inhomogeneous (possibly exploding) Markov process with values in 
$(E\times \nn)\cup\{(\triangle,\infty)\}$
with generator $\cL_t$, absorbed at $(\triangle,\infty)$ if this point is reached
and set to $(\triangle,\infty)$ after explosion if it explodes, where 
$$
\cL_t \Psi(y,n)= \kappa \Lambda(y) \int_H \int_{(E\times \nn)\cup\{(\triangle,\infty)\}} 
\Big[ \Psi(h(y,y^*,z),n+n^*+1) - \Psi(y,n)\Big] G_t(\dd y^*,\dd n^*) \frac{\nu_{y,y^*}(\dd z)}{\kappa}
$$
for all $t\geq 0$, all $\Psi \in B_b((E\times \nn)\cup \{(\triangle,\infty)\})$, 
all $y\in E$ and all $n \in \nn$.

\vip

{\it Step 2.} Here we handle a preliminary computation: for all $y,y^*\in E$, we have
\begin{equation}\label{xxx}
A(y,y^*)=\int_H \Big[\Lambda(h(y,y^*,z^*))-\Lambda(y)\Big] \nu_{y,y^*}(\dd z)\leq \kappa(1+e_0).
\end{equation}
Writing $y=(m,v)$, $y^*=(m^*,v^*)$ and recalling Subsection \ref{tlf}, $A(y,y^*)$ equals
$$
(1+e_0)\int_H \Big[|\vs(v,v^*,z)|^\gamma-|v|^\gamma \Big]\nu_{y,y^*}(\dd z)= 
(1+e_0)q(v,v^*)\intS [|v'(v,v^*,\sigma)|^\gamma-|v|^\gamma]
\beta_{v,v^*}(\sigma)\dd \sigma.
$$
But $|v'(v,v^*,\sigma)|\leq |v|+|v^*|$, see \eqref{vprime}, so that
$|v'(v,v^*,\sigma)|^\gamma -|v|^\gamma \leq |v^*|^\gamma$, whence
$$
A(y,y^*)\leq \kappa (1+e_0)q(v,v^*)|v^*|^\gamma=\kappa(1+e_0)
\frac{|v-v^*|^\gamma|v^*|^\gamma}{(1+|v|^\gamma)(1+|v^*|^2)}\leq \kappa(1+e_0),
$$
because $|v-v^*|^\gamma |v^*|^\gamma \leq (|v|^\gamma+|v^*|^\gamma) |v^*|^\gamma \leq 
|v|^\gamma(1+|v^*|^2)+(1+|v^*|^2)=(1+|v|^\gamma)(1+|v^*|^2)$.

\vip

{\it Step 3.} We now prove that $(Y_t,N_t)_{t\geq 0}$ actually does not explode nor reach the cemetery point,
that  $\E[N_t] \leq \exp(\kappa \intot \E[\Lambda(Y_s)] \dd s) -1$
and that $\E[\Lambda(Y_t)]\leq \E[\Lambda(Y_0)] \exp(\kappa (1+e_0)t)$.

\vip

For $A \in \nn_*$, we introduce $\zeta_A=\inf\{t\geq 0 : N_t\geq A\}$.
The process  $(Y_t,N_t)_{t\geq 0}$ does not explode nor reach the cemetery point during $[0,\zeta_A)$,
so that we can write, with $\Psi(y,n)=n\land A$, (recall that $N_0=0$ and that $t\mapsto N_t$ is a.s.
non-decreasing),
$$
\E[N_t \land A] = \E[N_{t \land \zeta_A}\land A]=\E\Big[\int_0^{t \land \zeta_A} \cL_s\Psi(Y_s,N_s)\dd s\Big]
=\intot \E\Big[\indiq_{\{N_s < A\}}\cL_s\Psi(Y_s,N_s)\Big] \dd s.
$$
Since $0\leq (n+n^*+1)\land A - n \land A \leq  1+n^*\land A$, we deduce that
$$
0\leq \cL_s \Psi(y,n) \leq \kappa \Lambda(y) \int_{(E\times \nn)\cup\{(\triangle,\infty)\}} [1+n^*\land A] 
G_s(\dd y^*,\dd n^*)= \kappa \Lambda(y) (1+\E[N_s\land A]),
$$ 
because $(Y_s,N_s)$ is 
$G_s$-distributed. We thus find
$$
\E[N_t\land A]\leq \kappa \int_0^t\E[\Lambda(Y_s)\indiq_{\{N_s<A\}}] \E[1+N_s\land A]\dd s,
$$
whence, by the Gronwall lemma,
\begin{equation}\label{jab1}
\E[N_t\land A] \leq \exp\Big(\kappa \intot \E[\Lambda(Y_s)\indiq_{\{N_s<A\}}] \dd s\Big) -1 .
\end{equation}

We next choose $\Psi(y,n)=\Lambda(y)\indiq_{\{n<A\}}$ and write, as previously,
$$
\E[\Lambda(Y_t)\indiq_{\{N_t <A\}}] = \E[\Lambda(Y_{t \land \zeta_A})\indiq_{\{N_{t\land \zeta_A} <A\}}]
= \E[\Lambda(Y_0)]+\E\Big[\int_0^{t \land \zeta_A} \cL_s\Psi(Y_s,N_s)\dd s\Big],
$$
whence
$$
\E[\Lambda(Y_t)\indiq_{\{N_t <A\}}] =\E[\Lambda(Y_0)]+\intot \E\Big[\indiq_{\{N_s <A\}}\cL_s\Psi(Y_s,N_s)\Big] \dd s.
$$
But $\cL_s\Psi(y,n)$ equals
\begin{align*}
& \Lambda(y) \int_H \int_{(E\times \nn)\cup\{(\triangle,\infty)\}} \!\!
\Big[\Lambda(h(y,y^*,z^*))\indiq_{\{n+n^*+1<A\}}- \Lambda(y)\indiq_{\{n<A\}}\Big] G_s(\dd y^*,\dd n^*)\nu_{y,y^*}(\dd z)\\
\leq &  \Lambda(y) \indiq_{\{n<A\}} \int_H \int_{E\times \nn}
\Big[\Lambda(h(y,y^*,z^*))-\Lambda(y)\Big] G_s(\dd y^*,\dd n^*)\nu_{y,y^*}(\dd z),
\end{align*}
whence $\cL_s\Psi(y,n) \leq  \kappa(1+e_0) \indiq_{\{n<A\}}\Lambda(y)$ by \eqref{xxx}
and since $G_s(E\times \nn)\leq 1$.
Finally, we have checked that
$\E[\Lambda(Y_t)\indiq_{\{N_t <A\}}] \leq \E[\Lambda(Y_0)]
+\kappa(1+e_0) \intot \E[\indiq_{\{N_s <A\}}\Lambda(Y_s)] \dd s,$
whence
\begin{equation}\label{jab2}
\E[\Lambda(Y_t)\indiq_{\{N_t <A\}}]\leq \E[\Lambda(Y_0)] \exp(\kappa(1+e_0) t).
\end{equation}

Gathering \eqref{jab1} and \eqref{jab2} and letting $A$ increase to infinity, we first conclude
that $\E[N_t]<\infty$ for all $t\geq 0$. In particular, $N_t<\infty$ a.s. for all $t\geq 0$, and 
the process $(Y_t,N_t)_{t\geq 0}$ does a.s. not explode and never reach $(\triangle,\infty)$.
Consequently, $\E[\Lambda(Y_t)]=\lim_{A\to \infty} \E[\Lambda(Y_t)\indiq_{\{N_t <A\}}]
\leq \E[\Lambda(Y_0)] \exp(\kappa (1+e_0)t)$ by \eqref{jab2}.
Finally, we easily conclude from \eqref{jab1} that $\E[N_t] \leq \exp(\kappa \intot \E[\Lambda(Y_s)] \dd s) -1$.

\vip

{\it Step 4.} By Step 3, we know that $G_t$ (which is the law of $(Y_t,N_t)$)
is actually supported by $E\times\nn$ for all $t\geq 0$. Hence Algorithm \ref{algo} a.s. stops.
The process $(Y_t,N_t)_{t\geq 0}$ is thus an inhomogeneous Markov with generator $\ctL_t$ defined, 
for $\Psi \in C_b(E\times\nn)$, by
$$\ctL_t \Psi(y,n)= \Lambda(y) \int_H \int_{E\times \nn}
\Big[ \Psi(h(y,y^*,z),n+n^*+1) - \Psi(y,n)\Big] G_t(\dd y^*,\dd n^*) \nu_{y,y^*}(\dd z)
$$
and we thus have
$$
\int_{E\times\nn}\Psi(y,n)G_t(\dd y,\dd n)= \int_{E\times\nn}\Psi(y,n)G_0(\dd y,\dd n)+\intot \int_{E\times \nn} 
\ctL_s \Psi(y,n) G_s(\dd y,\dd n) \dd s.
$$
Let now $F_t \in \cP(E)$ be the law of $Y_t$ (so $F_t$ is the first marginal of $G_t$).
It starts from $F_0$ and solves (A). Indeed, $\int_E |v|^\gamma F_t(\dd m,\dd v)\leq 
\int_E \Lambda(y)F_t(\dd y)=\E[\Lambda(Y_t)]$ is locally bounded by Step 3 
and for all $\Phi \in C_b(E)$,
applying the above equation with $\Psi(y,n)=\Phi(y)$, we find 
$\ctL_t \Psi(y,n)=  \Lambda(y) \int_H \int_{E} [ \Phi(h(y,y^*,z)) - \Phi(y)] F_t(\dd y^*) \nu_{y,y^*}(\dd z)$,
so that 
$$
\int_{E}\Phi(y)F_t(\dd y)\!=\!\! \int_{E}\Phi(y)F_0(\dd y)+\!\intot \int_{E}\int_E\int_H 
\Lambda(y)\Big[ \Phi(h(y,y^*,z)) - \Phi(y)\Big] \nu_{y,y^*}(\dd z)F_s(\dd y^*)F_s(\dd y) \dd s
$$
as desired. Finally, we have already seen in Step 3 that
$\E[N_t]\leq \exp(\kappa \intot \E[\Lambda(Y_s)] \dd s) -1=\exp(\kappa \intot \int_E \Lambda(y)F_s(\dd y) 
\dd s) -1$.
We have proved Proposition \ref{mr}, as well as the existence part of Proposition \ref{beeu}.
\hfill $\square$

\section{Series expansion}\label{sps}

The goal of this section is to prove Proposition \ref{psum}. 
We thus consider $F_0 \in \cP(E)$ such that $\int_E \Lambda(y)F_0(\dd y)
=(1+e_0)\int_E(1+|v|^\gamma)F_0(\dd m,\dd v)<\infty$. To shorten notation, we set 
$J_\Upsilon=J_\Upsilon(F_0)$. 

\vip

{\it Step 1.} Here we check that for all $\Upsilon \in \cT$, all $t\geq 0$, 
$C_\Upsilon(t)=\intot \int_E \Lambda(y)
J_\Upsilon(\dd s,\dd y)<\infty$. We work by induction. First, since 
$J_\circ(\dd t,\dd y)=\delta_0(\dd t)F_0(\dd y)$,
we find that $C_\circ(t)=\int_E \Lambda(y) F_0(\dd y)$ for all $t\geq 0$, which is finite by assumption.
Next, we fix $t\geq 0$, 
$\Upsilon \in \cT\setminus\{\circ\}$, we consider $\Upsilon_\ell$ and $\Upsilon_r$ as in the statement,
we assume by induction that $C_{\Upsilon_\ell}(t)<\infty$ and $C_{\Upsilon_r}(t)<\infty$ and prove that
$C_\Upsilon(t)<\infty$. We start from
\begin{align*}
C_\Upsilon(t)=&\intot \int_E \Lambda(x)Q(\Gamma_s(J_{\Upsilon_\ell}),\Gamma_s(J_{\Upsilon_r}))(\dd x) \dd s\\
=&\intot \int_E\int_E\int_H \Lambda(h(y,y^*,z)) \Lambda(y)\nu_{y,y^*}(\dd z) \Gamma_s(J_{\Upsilon_r})(\dd y^*)  
\Gamma_s(J_{\Upsilon_\ell})(\dd y)\dd s\\
=& \intot \int_E\int_E\int_H \int_0^s \int_0^s \Lambda(h(y,y^*,z))\Lambda(y) \nu_{y,y^*}(\dd z) \\
& \hskip5cm J_{\Upsilon_r}(\dd u^*,\dd y^*)e^{-\kappa\Lambda(y^*)(s-u^*)}J_{\Upsilon_\ell}(\dd u,\dd y) 
e^{-\kappa \Lambda(y)(s-u)}
\dd s.
\end{align*}
But it follows from \eqref{xxx} that $\int_H\Lambda(h(y,y^*,z))\nu_{y,y^*}(\dd z) \leq \kappa (1+e_0+\Lambda(y))\leq 
2\kappa \Lambda(y)$,
whence
\begin{align*}
C_\Upsilon(t)\leq & 2\kappa  \intot \int_E\int_E\int_0^s \int_0^s \Lambda^2(y) 
J_{\Upsilon_r}(\dd u^*,\dd y^*) J_{\Upsilon_\ell}(\dd u,\dd y) e^{-\kappa \Lambda(y)(s-u)}\dd s\\
\leq & 2\kappa C_{\Upsilon_r}(t) \intot \int_E\int_0^s \Lambda^2(y) 
J_{\Upsilon_\ell}(\dd u,\dd y) e^{-\kappa \Lambda(y)(s-u)}\dd s.
\end{align*}
We finally used that for $s\in [0,t]$, $\int_E\int_0^sJ_{\Upsilon_r}(\dd u^*,\dd y^*) \leq 
\int_E\int_0^t \Lambda(y^*)J_{\Upsilon_r}(\dd u^*,\dd y^*) =C_{\Upsilon_r}(t)$.
Next, by the Fubini theorem,
\begin{align*}
C_\Upsilon(t)\leq & 2\kappa  C_{\Upsilon_r}(t) \intot \int_E J_{\Upsilon_\ell}(\dd u,\dd y) \int_u^t \Lambda^2(y)
e^{-\kappa \Lambda(y)(s-u)}\dd s\leq 2  C_{\Upsilon_r}(t) \intot \int_E  \Lambda(y) J_{\Upsilon_\ell}(\dd u,\dd y),
\end{align*}
so that $C_\Upsilon(t)\leq 2 C_{\Upsilon_r}(t)C_{\Upsilon_\ell}(t) <\infty$ as desired.

\vip

{\it Step 2.} We deduce from Step 1 that for all $\Upsilon \in \cT$,
$$
t\mapsto \int_E \Lambda(y)\Gamma_t(J_\Upsilon)(\dd y)=\intot \int_E \Lambda(y)e^{-\Lambda(y)(t-s)}
J_\Upsilon(\dd s,\dd y)\leq C_\Upsilon(t)
$$
is locally bounded.

\vip

{\it Step 3.} We fix $k\in \nn_*$ and denote by $\cT_k\subset \cT$ the finite set of all ordered binary 
trees with at most
$k$ nodes. We introduce $F_t^k=\sum_{\Upsilon \in \cT_k} \Gamma_t(J_\Upsilon)$. By Step 2, we know that
$t\mapsto \int_E \Lambda(y) F_t^k(\dd y)$ is locally bounded. We claim that
for all $\Phi \in C_b(E)$, all $t\geq 0$,
\begin{align}\label{todiff}
\int_E \Phi(y) F_t^k(\dd y) =& \int_E \Phi(y)e^{-\kappa \Lambda(y)t}F_0(\dd y) \\
&+\sum_{\Upsilon \in \cT_k\setminus\{\circ\}} \intot \int_E \Phi(y)e^{-\kappa \Lambda(y)(t-s)} 
Q(\Gamma_s(J_{\Upsilon_\ell}),
\Gamma_s(J_{\Upsilon_r}))(\dd y)\dd s,\notag
\end{align}
whence in particular $F_0^k=F_0$.
Indeed, we first observe that
$$
\int_E \Phi(y) \Gamma_t(J_\circ)(\dd y) = \intot \int_E \Phi(y)e^{-\kappa \Lambda(y)(t-s)}J_\circ(\dd s,\dd y)
=\int_E \Phi(y)e^{-\kappa \Lambda(y)t}F_0(\dd y)
$$ 
and then that,
for $\Upsilon \in \cT_k\setminus\{\circ\}$,
\begin{align*}
\int_E \Phi(y)\Gamma_t(J_\Upsilon)(\dd y)=& \intot \int_E \Phi(y)e^{-\kappa \Lambda(y)(t-s)} J_{\Upsilon}(\dd s,\dd y).
\end{align*}
Since $J_{\Upsilon}(\dd s,\dd y)=Q(\Gamma_s(J_{\Upsilon_\ell}),
\Gamma_s(J_{\Upsilon_r}))(\dd y)\dd s$ by definition, the result follows.

\vip

{\it Step 4.} Differentiating \eqref{todiff}, we find that for all $\Phi\in C_b(E)$, all $t\geq 0$,
\begin{equation}\label{diff}
\frac d{dt}\int_E \Phi(y) F_t^k(\dd y) = - \kappa \int_E \Lambda(y) \Phi(y) F_t^k(\dd y) + 
\sum_{\Upsilon \in \cT_k\setminus\{\circ\}} \int_E \Phi(y)Q(\Gamma_t(J_{\Upsilon_\ell}),\Gamma_t(J_{\Upsilon_r}))(\dd y).  
\end{equation}
The differentiation is easily justified, using that $\Phi$ is bounded, that 
$t\mapsto \int_E \Lambda(y) F_t^k(\dd y)$ is locally bounded, as well as 
$t\mapsto  \int_E \Lambda(y) \Gamma_t(J_\Upsilon)(\dd y)$ for all $\Upsilon \in \cT$,
and that for $F,G\in\cM(E)$,
$$
Q(F,G)(E) = \kappa G(E) \int_E \Lambda(y)F(\dd y).
$$

{\it Step 5.} Here we verify that $\sup_{[0,\infty)} F_t^k(E) \leq 1$ and that 
$\sup_{k\geq 1} \sup_{[0,T]} \int_E \Lambda(y) F_t^k(\dd y) <\infty$ for all $T>0$. 
First observe that if $\Phi\in C_b(E)$ is nonnegative, then
\begin{align*}
\sum_{\Upsilon \in \cT_k\setminus\{\circ\}} \int_E \Phi(y)Q(\Gamma_t(J_{\Upsilon_\ell}),\Gamma_t(J_{\Upsilon_r}))(\dd y)
\leq& \sum_{\Upsilon_1 \in \cT_k, \Upsilon_2 \in \cT_k} 
\int_E \Phi(y)Q(\Gamma_t(J_{\Upsilon_1}),\Gamma_t(J_{\Upsilon_2}))(\dd y)\\
=&\int_E \Phi(y)Q(F_t^k,F_t^k)(\dd y).
\end{align*}
We used that the map $\Upsilon \mapsto (\Upsilon_\ell,\Upsilon_r)$ 
is injective from $\cT_k\setminus\{\circ\}$ into $\cT_k\times\cT_k$, as well as 
the bilinearity of $Q$. Consequently, by \eqref{diff},
\begin{align}\label{yyy}
\frac d{dt}\int_E \Phi(y) F_t^k(\dd y) \leq& - \kappa \int_E \Lambda(y) \Phi(y) F_t^k(\dd y) + 
\int_E \Phi(y)Q(F_t^k,F_t^k)(\dd y)\\
=& \int_E \int_E \cB\Phi(y,y^*)  F_t^k(\dd y^*)  F_t^k(\dd y) +  \kappa (F_t^k(E)-1)
\int_E \Lambda(y) \Phi(y) F_t^k(\dd y).\notag
\end{align}
For the last equality, we used that for any $F,G \in \cM(E)$, we have
\begin{equation}\label{zzz}
\int_E\int_E\cB \Phi(y,y^*) G(\dd y^*)F(\dd y)= \int_E \Phi(y)Q(F,G)(\dd y)-\kappa G(E)
\int_E \Lambda(y) \Phi(y) F(\dd y).
\end{equation}

Applying \eqref{yyy} with $\Phi=1$, we see that 
$$
\frac d {dt} F_t^k(E) \leq 
\kappa (F_t^k(E)-1)\int_E \Lambda(y) F_t^k(\dd y).
$$
Since $F_0^k(E)=F_0(E)=1$ and since $t\mapsto \int_E \Lambda(y)F_t^k(\dd y)$ is locally bounded,
we conclude that $F_t^k(E)\leq 1$ for all $t\geq 0$ (because 
$(d/dt) [(F_t^k(E)-1)\exp(- \kappa \intot \int_E \Lambda(y) F_s^k(\dd y)\dd s)] \leq 0$).

\vip

Applying next \eqref{yyy} with $\Phi=\Lambda$ and using that $F_t^k(E)\leq 1$, we find that
\begin{align*}
\frac d {dt} \int_E \Lambda(y) F_t^k(\dd y) 
\leq&  \int_E \int_E \cB\Lambda (y,y^*)  F_t^k(\dd y^*)  F_t^k(\dd y)
\leq \kappa(1+e_0) F_t^k(E)\int_E \Lambda(y) F_t^k(\dd y),
\end{align*}
because $\cB \Lambda(y,y^*)=\Lambda(y)\int_H[\Phi(h(y,y^*,z))-\Phi(y)]\nu_{y,y^*}(\dd z)\leq \kappa(1+e_0) 
\Lambda(y)$, see 
\eqref{xxx}. Since $F_0^k=F_0$ and since, again,  $F_t^k(E)\leq 1$, we conclude that
$\int_E \Lambda(y) F_t^k(\dd y) \leq [\int_E \Lambda(y) F_0(\dd y)]\exp(\kappa (1+e_0)t)$.

\vip

{\it Step 6.} By Step 5, the series of nonnegative measures 
$F_t=\sum_{\Upsilon \in \cT} \Gamma_t(J_\Upsilon)$ converges, satisfies 
$F_t(E)\leq 1$, and we know that $t\mapsto \int_E \Lambda(y)F_t(\dd y)$ is locally bounded.
Passing to the limit in the time-integrated version of \eqref{diff}, we find that
for all $\Phi \in C_b(E)$, all $t\geq 0$,
\begin{align}\label{ttac}
\int_E \Phi(y) F_t(\dd y) =&\int_E \Phi(y) F_0(\dd y) - \kappa \intot \int_E \Lambda(y) 
\Phi(y) F_s(\dd y)\dd s \\
& + \sum_{\Upsilon \in \cT\setminus\{\circ\}} \intot\int_E \Phi(y)
Q(\Gamma_s(J_{\Upsilon_\ell}),\Gamma_s(J_{\Upsilon_r}))(\dd y)\dd s. \notag
\end{align}
To justify the limiting procedure, it suffices to use that $t\mapsto \int_E \Lambda(y)F_t(\dd y)$ 
is locally bounded, as well as $t\mapsto \sum_{\Upsilon \in \cT\setminus\{\circ\}}
Q(\Gamma_t(J_{\Upsilon_\ell}),\Gamma_t(J_{\Upsilon_r}))(E)=
\sum_{\Upsilon_1 \in \cT, \Upsilon_2 \in \cT} Q(\Gamma_t(J_{\Upsilon_1}),\Gamma_t(J_{\Upsilon_2}))(E)$, which equals 
$Q(\sum_{\Upsilon_1 \in \cT}\Gamma_t(J_{\Upsilon_1}),\sum_{\Upsilon_2 \in \cT}\Gamma_t(J_{\Upsilon_2}))(E)=Q(F_t,F_t)(E)
=\kappa F_t(E) \int_E \Lambda(y)F_t(\dd y)$. We used that 
the map $\Upsilon \mapsto (\Upsilon_\ell,\Upsilon_r)$ 
is bijective from $\cT\setminus\{\circ\}$ into $\cT\times\cT$, as well as 
the bilinearity of $Q$.
By the same way, \eqref{ttac} rewrites as
\begin{align*}
\int_E \Phi(y) F_t(\dd y) =&\int_E \Phi(y) F_0(\dd y) - \kappa \intot \int_E \Lambda(y) \Phi(y) F_s(\dd y)\dd s
 + \intot\int_E \Phi(y) Q(F_s,F_s)(\dd y)\dd s\\
=&\int_E \Phi(y) F_0(\dd y) + \intot \int_E \int_E \cB\Phi(y,y^*) F_s(\dd y^*)F_s(\dd y)\dd s\\
 &+\kappa \intot (F_s(E)-1) \Big(\int_E \Lambda(y) \Phi(y) F_s(\dd y) \Big)\dd s,
\end{align*}
see \eqref{zzz}.
To conclude that $(F_t)_{t\geq 0}$ solves (A), it only remains to verify that 
$t\mapsto \int_E |v|^\gamma F_t(\dd m,\dd v)$ is locally bounded, which follows from the fact that
$\int_E |v|^\gamma F_t(\dd m,\dd v)\leq \int_E \Lambda(y) F_t(\dd y)$, and 
that $F_t(E)=1$ for all $t\geq 0$. 
Applying the previous equation with $\Phi=1$ (for which $\cB\Phi=0$), 
we find that $F_t(E)=1+ \intot (F_s(E)-1) \alpha_s \dd s$,
where $\alpha_s= \kappa \int_E \Lambda(y) F_s(\dd y)$ is locally bounded. Hence $F_t(E)=1$ for all $t\geq 0$
by the Gronwall lemma.
The proof of Proposition \ref{psum} is complete. \hfill $\square$

\section{Well-posedness of (A)}\label{swp}

We have already checked (twice) the existence part of Proposition \ref{beeu}. We now turn to uniqueness.
Let us consider two solutions $F$ and $G$ to (A) with $F_0=G_0$. By assumption,
we know that
$\alpha_t=\int_E \Lambda(y) (F_t+G_t)(\dd y)=(1+e_0)\int_E (1+|v|^\gamma)(F_t+G_t)(\dd m,\dd v)$ 
is locally bounded. Hence, setting 
$\e^M_t=\int_E \Lambda(y)\indiq_{\{\Lambda(y)\geq M\}} (F_t+G_t)(\dd y)$, we have
$\lim_{M\to \infty} \intot \e^M_s \dd s=0$ for all $t\geq 0$.

\vip

We use the total variation distance $u_t=||F_t-G_t||_{TV}=\sup \{u_t^\Phi : \Phi \in C_b(E),
||\Phi||_\infty \leq 1\}$
where $u_t^\Phi=\int_E \Phi(y)(F_t-G_t)(\dd y)$. We also have $u_t=\int_E |F_t-G_t|(\dd y)$, where 
for $\mu$ a 
finite signed measure on $E$, $|\mu|=\mu_++\mu_-$ with the usual definitions of $\mu_+$ and $\mu_-$.

\vip

We fix $\Phi \in C_b(E)$ such that $||\Phi||_\infty \leq 1$ and we use Definition \ref{GE} to write 
$$
\frac d{dt}u_t^\Phi= \int_E\int_E \cB\Phi(y,y^*) (F_t(\dd y^*)F_t(\dd y)-G_t(\dd y^*)G_t(\dd y))
=A_t^\Phi+B_t^\Phi,
$$
where $A_t^\Phi=\int_E\int_E \cB\Phi(y,y^*) (F_t-G_t)(\dd y^*)F_t(\dd y)$ 
and $B_t^\Phi=\int_E\int_E \cB\Phi(y,y^*) G_t(\dd y^*)(F_t-G_t)(\dd y)$.

\vip

Using only that $|\cB\Phi(y,y^*)|\leq 2\kappa ||\Phi||_\infty \Lambda(y) \leq 2\kappa \Lambda(y)$, we get
$$
A_t^\Phi \leq 2\kappa \int_E\Lambda(y) F_t(\dd y)\int_E |F_t-G_t|(\dd y^*)\leq
2\kappa \alpha_t ||F_t-G_t||_{TV}=2\kappa \alpha_t u_t.
$$

We next recall that $||\Phi||_\infty\leq 1$ and that $\int_E G_t(\dd y^*)=1$ and we write
\begin{align*}
B^\Phi_t = & \int_E\int_E\int_H \Lambda(y)\Phi(h(y,y^*,z)) \nu_{y,y^*}(dz)  G_t(\dd y^*)(F_t-G_t)(\dd y)
- \kappa  \int_E \Lambda(y)\Phi(y)(F_t-G_t)(\dd y)\\
\leq & \kappa \int_E \Lambda(y)|F_t-G_t|(\dd y)-\kappa  \int_E \Lambda(y)\Phi(y)(F_t-G_t)(\dd y).
\end{align*}
Since now $|F_t-G_t|(\dd y)-\Phi(y)(F_t-G_t)(\dd y)$ is a nonnegative measure bounded by
$2(F_t+G_t)(\dd y)$, we may write, for any
$M\geq 1$,
\begin{align*}
B^\Phi_t \leq &  \kappa M \int_E [|F_t-G_t|(\dd y)-\Phi(y)(F_t-G_t)(\dd y)] + 2\kappa 
\int_E \Lambda(y)\indiq_{\{\Lambda(y)\geq M\}} (F_t+G_t)(\dd y)\\
=& \kappa M u_t - \kappa M u_t^\Phi + 2\kappa \e^M_t.
\end{align*}

All this proves that $(d/dt)u_t^\Phi \leq 2\kappa \alpha_t u_t+\kappa M u_t - \kappa M u_t^\Phi + 2\kappa \e^M_t$,
whence
$$
\frac d {dt} (u_t^\Phi e^{\kappa M t}) \leq  [2\kappa \alpha_t u_t+\kappa M u_t + 2\kappa \e^M_t]e^{\kappa M t}.
$$
Integrating in time (recall that $u_0^\Phi=0$) and taking the supremum over $\Phi \in C_b(E)$ such that 
$||\Phi||_\infty\leq 1$, we find
$u_te^{\kappa M t} \leq \intot [(2\kappa \alpha_s + \kappa M)u_s +2\kappa \e^M_s]e^{\kappa M s} \dd s.$

\vip

Recall the following generalized Gronwall lemma: if we have three locally bounded nonnegative functions
$v,g,h$ such that $v_t \leq \intot (h_s v_s+g_s)\dd s$ for all $t\geq 0$, then
$v_t \leq \intot g_s \exp(\int_s^t h_u \dd u) \dd s$. Applying this result with 
$v_t=u_te^{\kappa M t}$, $g_t=2\kappa \e^M_te^{\kappa M t}$ and $h_t=2\kappa \alpha_t + \kappa M$, we get
$$
u_t e^{\kappa M t} \leq  2\kappa  \intot  \e^M_s \exp\Big(
\kappa M s+2 \kappa \int_s^t \alpha_u \dd u + \kappa M(t-s) \Big) \dd s,
$$
so that $u_t \leq  2\kappa \intot  \e^M_s \exp(2 \kappa \int_s^t \alpha_u\dd u) \dd s$.
Recalling that $\alpha$ is locally bounded and that $\intot \e^M_s \dd s$ tends to $0$ as $M\to \infty$,
we conclude that $u_t=0$, which was our goal. The proof of Proposition \ref{beeu} is now complete.
\hfill $\square$

\vip

We end this section with the

\begin{proof}[Proof of Remark \ref{dec}]
We fix $A\geq 1$ and apply \eqref{gew} with $\Phi_A(m,v)=|v|^2\land A$, which belongs to $C_b(E)$.
With the notation $y=(m,v)$ and $y^*=(m^*,v^*)$, we find
\begin{align*}
\cB\Phi_A(y,y^*)=&\Lambda(v) \int_H[|\vs(v,v^*,z)|^2\land A - |v|^2\land A]\nu_{y,y^*}(\dd z)\\
=&\Lambda(v)q(v,v^*)
\intS [|v'(v,v^*,\sigma)|^2\land A- |v|^2\land A] \beta_{v,v^*}(\sigma)\dd\sigma\\
=&\kappa(1+e_0)\frac{|v-v^*|^\gamma}{1+|v^*|^2} 
\Big[\kappa^{-1}\intS (|v'(v,v^*,\sigma)|^2\land A ) \beta_{v,v^*}(\sigma)\dd \sigma - |v|^2\land A\Big]\\
\leq & \kappa(1+e_0)\frac{|v-v^*|^\gamma}{1+|v^*|^2} 
\Big[ \Big(\kappa^{-1}\intS|v'(v,v^*,\sigma)|^2\beta_{v,v^*}(\sigma)\dd \sigma\Big)
\land A - 
|v|^2\land A\Big].
\end{align*}
But a simple computation, recalling \eqref{vprime}
and using that 
$$
\frac{|v-v^*|}\kappa \intS \sigma \beta_{v,v^*}(\sigma)\dd \sigma = 
\frac{|v-v^*|}\kappa \intS \sigma b({\textstyle{\langle\frac{v-v^*}{|v-v^*|},\sigma\rangle}})\dd \sigma
= c (v-v^*)
$$
where $c= 2\pi\kappa^{-1}\int_0^\pi \sin \theta \cos\theta b(\cos\theta)\dd \theta \in [-1,1]$ 
(recall \eqref{beta}) shows that 
\begin{equation}\label{ccc}
\kappa^{-1}\intS |v'(v,v^*,\sigma)|^2\beta_{v,v^*}(\sigma)\dd \sigma=
\frac{1+c}2|v|^2+\frac{1-c}2|v^*|^2=(1-\alpha)|v|^2+\alpha|v^*|^2, 
\end{equation}
where $\alpha=(1-c)/2 \in [0,1]$. Hence, 
$$
\cB\Phi_A(y,y^*)+\cB\Phi_A(y^*,y) \leq \kappa(1+e_0)|v-v^*|^\gamma G_A(|v|^2,|v^*|^2),
$$
where $G_A(x,x^*)= (1+x^*)^{-1}[((1-\alpha)x+\alpha x^*)\land A - x\land A] + 
(1+x)^{-1}[((1-\alpha)x^*+\alpha x)\land A - x^*\land A]$.
One can check that $G_A(x,x^*)\leq 0$ if $x\lor x^*\leq A$ and it always holds true that
$G_A(x,x^*)\leq (1+x^*)^{-1}\alpha x^*+ (1+x)^{-1}\alpha x \leq 2$. At the end, 
$G_A(x,x^*) \leq 2(\indiq_{\{x>A\}}+\indiq_{\{x^*>A\}})$. Consequently, applying \eqref{gew} 
and using a symmetry argument,
\begin{align*}
\int_E (|v|^2&\land A)F_t(\dd y)=\!\int_E (|v|^2\land A)F_0(\dd y)
+ \intot \int_E\int_E \cB\Phi_A(y,y^*) F_s(\dd y^*)F_s(\dd y) \dd s\\
=&\!\int_E (|v|^2\land A)F_0(\dd y)
+ \frac12 \intot \int_E\int_E [\cB\Phi_A(y,y^*)+\cB\Phi_A(y^*,y)] F_s(\dd y^*)F_s(\dd y) \dd s\\
\leq & \int_E |v|^2F_0(\dd y)+ \kappa(1+e_0)
 \intot \int_E\int_E |v-v^*|^\gamma[\indiq_{\{|v|^2>A\}}+\indiq_{\{|v^*|^2>A\}}] F_s(\dd y^*)F_s(\dd y) \dd s.
\end{align*}
Letting $A\to \infty$ and using that $\intot  \int_E\int_E |v-v^*|^\gamma  F_s(\dd y^*)F_s(\dd y) \dd s<\infty$
(which follows from the fact that $\sup_{[0,t]} \int |v|^\gamma F_s(\dd y)<\infty$), we conclude that
$\int_E |v|^2 F_t(\dd m,\dd v)\leq \int_E |v|^2 F_0(\dd m,\dd v)$.
\end{proof}

\section{Relation between (A) and the Boltzmann equation}\label{sr}

It  remains to prove Proposition \ref{relat}. In the whole section, we consider
the solution $F$ to (A) starting from some $F_0 \in \cP(E)$
such that $\int_E [|v|^\gamma+m(1+|v|^{2+2\gamma})] F_0(\dd y)<\infty$.
We define the nonnegative measure $f_t$ on $\rr^3$ 
by $f_t(A)=\int_E m\indiq_{\{v \in A\}} F_t(\dd y)$ for all $A\in\cB(\rr^3)$ 
and we assume that $f_0 \in \cP(\rr^3)$ and
that $\int_{\rr^3}|v|^2 f_0(\dd v)=e_0$, where $e_0$ was used in Subsection \ref{tlf} to build the
coefficients of (A). We want to prove
that $f=(f_t)_{t\geq 0}$ is a weak solution to \eqref{be}.

\vip

The main difficulty is to establish properly the following estimate, of which the proof is postponed
at the end of the section.

\begin{lem}\label{ttc}
For any $T>0$, $\sup_{[0,T]}\int_E m (1+|v|^{2+\gamma})F_t(\dd m,\dd v)<\infty$. 
\end{lem}

Next, we handle a few preliminary computations.

\begin{rk}\label{tf}
(i) For all $\Phi\in C(E)$ of the form  $\Phi(m,v)=m\phi(v)$ with $\phi \in C(\rr^3)$,
using the notation $y=(m,v)$ and $y^*=(m^*,v^*)$, it holds that
$$
\cB \Phi(y,y^*)=mm^* \cA\phi(v,v^*)+\kappa m \Lambda(v)\phi(v)
\Big(\frac{m^*(1+|v^*|^2)}{1+e_0} -1\Big).
$$
(ii) Assume furthermore that there is $\alpha \geq 0$ such that for all $v\in \rd$, 
$|\phi(v)|\leq C (1+|v|^{2+\alpha})$. Then $|\cB\Phi(y,y^*)|\leq C [\Lambda(y)+\Lambda(y^*)][1+m(1+|v|^{2+\alpha})]
[1+m^*(1+|v^*|^{2+\alpha})]$.
\end{rk}

\begin{proof}
For (i), it suffices to write  
$$
\cB \Phi(y,y^*)=\Lambda(v)\int_H\Big[\frac{mm^*(1+|v^*|^2)}{1+e_0}\phi(v''(v,v^*,z))-m\phi(v)\Big] \nu_{y,y^*}
(\dd z)=\cB^1 \Phi(y,y^*)+\cB^2 \Phi(y,y^*), 
$$
where 
\begin{align*}
\cB^1\Phi(y,y^*)=&\Lambda(v)\frac{mm^*(1+|v^*|^2)}{1+e_0}\int_H
\Big[\phi(\vs(v,v^*,z))-\phi(v)\Big] \nu_{y,y^*}(\dd z)\\
=&\Lambda(v)q(v,v_*)\frac{mm^*(1+|v^*|^2)}{1+e_0}\intS
\Big[\phi(v'(v,v^*,\sigma))-\phi(v) \Big] \beta_{v,v^*}(\sigma) \dd \sigma,
\end{align*}
which equals $mm^*\cA\phi(v,v^*)$, and where 
\begin{align*}
\cB^2\Phi_A(y,y^*)=&\kappa \Lambda(v) \phi(v) \Big[\frac{mm^*(1+|v^*|^2)}{1+e_0} -m\Big].
\end{align*}

For point (ii), we first observe that $|\cA\phi(v,v^*)|\leq C |v-v^*|^\gamma(1+|v|^{2+\alpha}+|v^*|^{2+\alpha})$,
because $|v'(v,v^*,\sigma)|\leq |v|+|v^*|$, see \eqref{vprime}.
Using next that $|v-v^*|^\gamma\leq |v|^\gamma+|v^*|^\gamma$
and that $\Lambda(y)=\Lambda(v)=(1+e_0)(1+|v|^\gamma)$, we thus find
\begin{align*}
|\cB\Phi(y,y^*)|\leq & C mm^*(|v|^\gamma+|v^*|^\gamma)(1+|v|^{2+\alpha}+|v^*|^{2+\alpha})
+ C \Lambda(v) m(1+|v|^{2+\alpha})(1+m^*(1+|v^*|^2)),
\end{align*}
from which the conclusion easily follows.
\end{proof}

We now can give the

\begin{proof}[Proof of Proposition \ref{relat}]
For any $\phi\in C(\rr^3)$ such that $|\phi(v)|\leq C(1+|v|^2)$, we can apply \eqref{gew}
with $\Phi(m,v)=m\phi(v)$. To check it it properly, first apply \eqref{gew} with 
$\Phi(m,v)=[m\land A] \phi_A(v)$ with $\phi_A(v)=\phi(v)\land A \lor (-A)$ and let $A\to \infty$.
This essentially relies on the facts
that 

\vip

\noindent $\bullet$ $|\cB\Phi(y,y^*)|\leq C [\Lambda(y)+\Lambda(y^*)][1+m(1+|v|^2)][1+m^*(1+|v^*|^2)]$
by Remark \ref{tf}-(ii) (with $\alpha=0$), whence
$ |\cB\Phi(y,y^*)|\leq C
[1+|v|^\gamma+m(1+|v|^{2+\gamma})][1+|v^*|^\gamma+m^*(1+|v^*|^{2+\gamma})]$ and

\vip

\noindent $\bullet$ $t\mapsto \int_E [1\!+\!|v|^\gamma\!+\!m (1\!+\!|v|^{2+\gamma})]F_t(\dd m,\dd v)<\infty$ 
is locally bounded by Lemma \ref{ttc} and Definition \ref{GE}.

\vip

So, applying \eqref{gew} and using the formula of Remark \ref{tf}-(i), we find
\begin{align*}
\int_E m\phi(v)F_t(\dd y)=&\int_E m\phi(v)F_0(\dd y)+\intot\int_E\int_E mm^* \cA\phi(v,v^*) 
F_s(\dd y^*)F_s(\dd y)\dd s\\
&+\kappa \intot\int_E\int_E m \Lambda(v)\phi(v) \Big(\frac{m^*(1+|v^*|^2)}{1+e_0} -1\Big)
F_s(\dd y^*)F_s(\dd y)\dd s.
\end{align*}
This precisely rewrites, by definition of $f_t$,
\begin{align}\label{ww}
\int_{\rr^3} \phi(v)f_t(\dd v)=&\int_{\rr^3} \phi(v)f_0(\dd v)+ \intot \int_{\rr^3}\int_{\rr^3} \cA\phi(v,v^*) 
f_s(\dd v^*)f_s(\dd v) \dd s \\
&+ \kappa \intot  (\Theta_s-1) \Big(\int_{\rr^3} \Lambda(v)\phi(v)f_s(\dd v)\Big) \dd s, \nonumber
\end{align}
where $\Theta_t=(1+e_0)^{-1}\int_{\rr^3} (1+|v|^2)f_t(\dd v)$. 

\vip

But, with $\phi(v)=(1+e_0)^{-1}(1+|v|^2)$, recalling \eqref{ccc}, it holds that
\begin{equation}\label{coco}
\cA\phi(v,v^*)+\cA\phi(v^*,v)=\frac{\kappa}{1+e_0}\big[(1-\alpha)|v|^2+\alpha |v^*|^2 - |v|^2+
(1-\alpha)|v^*|^2+\alpha |v|^2 - |v^*|^2\big]=0.
\end{equation}
Hence applying \eqref{ww} and using a symmetry argument, we find
$$
\Theta_t = 1 + \kappa \intot  (\Theta_s-1) \Big(\int_{\rr^3} \Lambda(v)\phi(v)f_s(\dd v)\Big) \dd s.
$$
Hence $\Theta_t=1$ for all $t\geq 0$ by the Gronwall Lemma, 
because $\Theta_t=(1+e_0)^{-1}\int_E m(1+|v|^2)F_t(\dd m,\dd v)$ and 
$\int_{\rr^3} \Lambda(v)\phi(v)f_t(\dd v)= \int_E m(1+|v|^2)(1+|v|^\gamma)F_t(\dd m,\dd v)$ 
are locally bounded by Lemma \ref{ttc}.

\vip

Coming back to \eqref{ww}, we thus see that for all $\phi\in C_b(\rr^3)$, 
\begin{align*}
\int_{\rr^3} \phi(v)f_t(\dd v)=&\int_{\rr^3} \phi(v)f_0(\dd v)+ \intot \int_{\rr^3}\int_{\rr^3} \cA\phi(v,v^*) 
f_s(\dd v^*)f_s(\dd v) \dd s.
\end{align*}
To complete the proof, it only remains to prove that $f_t(\rd)=1$ for all $t\geq 0$, 
which follows from the choice
$\phi(v)=1$ (for which $\cA \phi(v,v_*)=0$), and to check that 
$\intrd |v|^2 f_t(\dd v)=e_0$ for all $t\geq 0$, which holds true because
$\intrd |v|^2 f_t(\dd v)=(1+e_0)\Theta_t - f_t(\rr^3)$.
\end{proof}

It only remains to prove Lemma \ref{ttc}.

\begin{proof}[Proof of Lemma \ref{ttc}] 
The proof relies on the series expansion
$F_t=\sum_{\Upsilon \in \cT} \Gamma_t(J_\Upsilon(F_0))$, see Proposition \ref{psum}.
We write $J_\Upsilon=J_\Upsilon(F_0)$ for simplicity. We will make use of the functions 
$\Phi_0(m,v)=m(1+|v|^{2})/(1+e_0)$, $\Phi_1(m,v)=m(1+|v|^{2+\gamma})$ and $\Phi_2(m,v)=m(1+|v|^{2+2\gamma})$.

\vip

{\it Step 1.} Here we verify that for all $\Upsilon \in \cT$, all $t\geq 0$,
$D_\Upsilon(t)= \intot\int_E \Phi_2(y) J_\Upsilon(\dd s,\dd y)<\infty$. We proceed by induction
as in the proof of Proposition \ref{psum}, Step 1. First, $D_\circ (t)=\int_E \Phi_2(y)F_0(\dd y)<\infty$
by assumption. Next, we fix $t\geq 0$, 
$\Upsilon \in \cT\setminus\{\circ\}$, we assume by induction that $D_{\Upsilon_\ell}(t)<\infty$ and 
$D_{\Upsilon_r}(t)<\infty$ and prove that $D_\Upsilon(t)<\infty$. We start from
\begin{align*}
D_\Upsilon(t)=&\intot \int_E\int_E\int_H \int_0^s \int_0^s \Phi_2(h(y,y^*,z))\Lambda(y) \nu_{y,y^*}(\dd z) \\
& \hskip5cm J_{\Upsilon_r}(\dd u^*,\dd y^*)e^{-\kappa\Lambda(y^*)(s-u^*)}J_{\Upsilon_\ell}(\dd u,\dd y) 
e^{-\kappa \Lambda(y)(s-u)}\dd s.
\end{align*}
But we see from Remark \ref{tf}-(ii) (with $\alpha=2\gamma$) that
\begin{align*}
\Lambda(y)\int_H \Phi_2(h(y,y^*,z))\nu_{y,y^*}(\dd z)=& \cB\Phi_2(y,y^*)+ \kappa \Lambda(y)\Phi_2(y)\\
\leq& C [\Lambda(y)\!+\!\Lambda(y^*)](1\!+\! \Phi_2(y))(1\!+\!\Phi_2(y^*)).
\end{align*}
Together with the Fubini theorem, this gives us 
\begin{align*}
D_\Upsilon(t)\leq &C \int_E \intot \int_E \intot (1+\Phi_2(y))(1+\Phi_2(y^*))J_{\Upsilon_r}(\dd u^*,\dd y^*)
J_{\Upsilon_\ell}(\dd u,\dd y) \\
&\hskip4cm\int_{u\lor u^*}^t [\Lambda(y)+\Lambda(y^*)]
e^{-\kappa\Lambda(y^*)(s-u^*)} e^{-\kappa \Lambda(y)(s-u)}\dd s\\
\leq & C \int_E \intot \int_E \intot (1+\Phi_2(y))(1+\Phi_2(y^*))
J_{\Upsilon_r}(\dd u^*,\dd y^*)J_{\Upsilon_\ell}(\dd u,\dd y)\\
=& C [J_{\Upsilon_\ell}([0,t]\times E)+D_{\Upsilon_\ell}(t)][J_{\Upsilon_r}([0,t]\times E)+D_{\Upsilon_r}(t)].
\end{align*}
We conclude by induction and since we already know from Step 1 of the proof of Proposition \ref{psum}
that $J_{\Upsilon}([0,t]\times E)\leq \intot\int_E \Lambda(y)J_\Upsilon(\dd s,\dd y) <\infty$
for all $\Upsilon \in \cT$.

\vip

{\it Step 2.} For $k \in \nn^*$, we define 
$F^k_t=\sum_{\Upsilon \in \cT_k} \Gamma_t(J_\Upsilon(F_0))$ as in the proof of Proposition \ref{psum}, Step 3.
We know that $F^k_0=F_0$ and that for all nonnegative $\Phi \in C_b(E)$, see \eqref{yyy} and recall that
$F^k_t(E)\leq1 $,
\begin{equation}\label{toex}
\int_E \Phi(y)F_t^k(\dd y) \leq \int_E \Phi(y)F_0(\dd y) +\intot\int_E\int_E\cB\Phi(y,y^*)F_s^k(\dd y^*)
F_s^k(\dd y)\dd s.
\end{equation}
Also, we immediately deduce from Step 1 that $t\mapsto \int_E \Phi_2(y)F_t^k(\dd y)$ is locally bounded,
as well as $t\mapsto \int_E \Lambda(y)F_t^k(\dd y)$, see Step 2 of the proof of Proposition \ref{psum}.
It is then easy to extend \eqref{toex} to any function $\Phi \in C(E)$ of the form
$\Phi(m,v)=m\phi(v)$, with $0\leq \phi(v)\leq C(1+|v|^{2+\gamma})$. 
This follows from the fact that, by Remark \ref{tf}-(ii) (with $\alpha=\gamma$),
$$
|\cB\Phi(y,y^*)| \leq C[\Lambda(y)+\Lambda(y^*)](1+\Phi_1(y))(1+\Phi_1(y^*))\leq C (1+\Lambda(y)+\Phi_2(y))
(1+\Lambda(y^*)+\Phi_2(y^*)).
$$

\vip

{\it Step 3.} We now verify that $\int_E \Phi_0(y)F_t^k(\dd y)\leq 1$ for all $t\geq 0$.
To this end, we apply \eqref{toex} with $\Phi=\Phi_0$ for which, by Remark \ref{tf}-(i), 
$$
\cB \Phi_0(y,y^*)= m m^*\cA\phi(v,v^*)+\kappa \Lambda(y)\Phi_0(y)[\Phi_0(y^*)-1 ],
$$
where $\phi(v)=(1+|v|^2)/(1+e_0)$. Using that $\cA\phi(v,v^*)+\cA\phi(v^*,v)=0$ (recall
\eqref{coco}) and a symmetry argument, we find that
\begin{align*}
\int_E \Phi_0(y)F_t^k(\dd y) \leq & \int_E \Phi_0(y)F_0(\dd y) +
\intot\int_E\int_E \kappa  \Lambda(y)\Phi_0(y)
[\Phi_0(y^*)-1 ]F_s^k(\dd y^*) F_s^k(\dd y)\dd s\\
=& 1 + \intot \Big(\int_E \kappa  \Lambda(y)\Phi_0(y) F_s^k(\dd y)\Big)\Big( \int_E \Phi_0(y)F_s^k(\dd y)-1  
\Big) \dd s.
\end{align*}
Setting $u_t=\int_E \Phi_0(y)F_t^k(\dd y)-1$ and $\alpha_t=\int_E \kappa  \Lambda(y)\Phi_0(y) F_t^k(\dd y)\geq 0$,
we know that $u$ and $\alpha$ are locally bounded by Step 2 (because $\Phi_0(y)+\Lambda(y)\Phi_0(y) 
\leq C \Phi_2(y)$) 
and that $u_t \leq \intot \alpha_s u_s \dd s$. This implies that $u_t\leq 0$ for all $t\geq 0$,
which was our goal.

\vip

{\it Step 4.} We finally apply \eqref{toex} with $\Phi=\Phi_1$. By Remark \ref{tf}, we see that,
with $\phi(v)=1+|v|^{2+\gamma}$,
$$
\cB\Phi_1(y,y^*)=mm^* \cA\phi(v,v^*)+\kappa \Lambda(y)\Phi_1(y)(\Phi_0(y^*)-1).
$$
Hence
\begin{align*}
\int_E \Phi_1(y)F_t^k(\dd y) \leq & \int_E \Phi_1(y)F_0(\dd y) 
+ \intot \int_E\int_E mm^* \cA\phi(v,v^*) F_s^k(\dd y^*)F^k_s(\dd y) \dd s\\
&+\intot\int_E\int_E \kappa  \Lambda(y)\Phi_1(y)[\Phi_0(y^*)-1 ]F_s^k(\dd y^*) F_s^k(\dd y)\dd s\\
\leq& \int_E \Phi_1(y)F_0(\dd y) + \intot \int_E\int_E mm^* \cA\phi(v,v^*) F_s^k(\dd y^*)F^k_s(\dd y) \dd s
\end{align*}
by Step 3. We next recall a Povzner lemma \cite{p} in the version found in 
\cite[Lemma 2.2-(i)]{MW}: for $\alpha>0$, setting $\phi_\alpha(v)=|v|^{2+\alpha}$, 
there is a $C_\alpha>0$ such that for all $v,v^* \in \rd$, 
$\cA\phi_\alpha(v,v^*)+\cA\phi_\alpha(v^*,v) \leq C_\alpha |v-v^*|^{\gamma} (|v||v^*|)^{1+\alpha/2}$.
Actually, the result of \cite{MW} is much stronger. Since $\phi=1+\phi_\gamma$ and since 
$\cA1=0$, we conclude that 
\begin{align*}
\cA\phi(v,v^*)+\cA\phi(v^*,v)  \leq & C |v-v^*|^\gamma (|v||v^*|)^{1+\gamma/2}\\
\leq& C|v|^{1+3\gamma/2}|v^*|^{1+\gamma/2}+C|v|^{1+\gamma/2}|v^*|^{1+3\gamma/2}\\
\leq& C (1+|v|^{2+\gamma})(1+|v^*|^2)+C(1+|v|^{2})(1+|v^*|^{2+\gamma}),
\end{align*}
so that
$$
mm^*[\cA\phi(v,v^*)+\cA\phi(v^*,v)] \leq C [\Phi_1(y)\Phi_0(y^*)+\Phi_1(y^*)\Phi_0(y)]
$$
Finally, using twice a symmetry argument,
\begin{align*}
\int_E \Phi_1(y)F_t^k(\dd y)
\leq &\int_E \Phi_1(y)F_0(\dd y) + C \intot\int_E \int_E \Phi_1(y)\Phi_0(y^*) F_s^k(\dd y^*)F^k_s(\dd y) \dd s\\
\leq & \int_E \Phi_1(y)F_0(\dd y) + C \intot\int_E\Phi_1(y)F^k_s(\dd y) \dd s
\end{align*}
by Step 3 again.
Hence $\int_E \Phi_1(y)F_t^k(\dd y) \leq  e^{Ct}\int_E \Phi_1(y)F_0(\dd y)$ by the Gronwall lemma.
It then suffices to let $k$ increase to infinity, by monotone convergence, to complete the proof.
\end{proof}

\end{document}